\documentclass[]{amsart}
\usepackage[english]{babel}
\usepackage[T1]{fontenc}
\usepackage[utf8]{inputenc}

\usepackage{amsmath,amsthm,mathrsfs,amsfonts,amssymb,mathtools}
\usepackage{amsaddr}
\usepackage{tikz}
\usepackage{tikz-cd}
\usepackage{todonotes}
\usepackage{enumitem}
\usepackage{graphicx}
\usepackage{hyperref}
\usepackage[capitalise]{cleveref}
\usepackage{float}
\usepackage{setspace}
\usepackage{listings}
\usepackage{color} 
\usepackage{graphicx}
\usepackage{xifthen}

\newtheorem{thm}{Theorem}[section] 

\newtheorem{prop}[thm]{Proposition}

\newtheorem{teo}[thm]{Theorem}
\newtheorem{mainteo}{Theorem}

\theoremstyle{definition}
\newtheorem{defn}[thm]{Definition}

\newtheorem{oss}[thm]{Remark}

\newtheorem{conj}[thm]{Conjecture}

\newcommand{\R}{\mathbb{R}}

\newcommand{\HH}{\mathbb{H}}

\newcommand{\vol}{\mathrm{vol}}

\NewDocumentCommand{\norm}{m O{}}{{\left\lVert#1\right\rVert_{#2}}}

\let\temp\phi
\let\phi\varphi
\let\varphi\temp

\let\temp\epsilon
\let\epsilon\varepsilon
\let\varepsilon\temp

\title[Bounded volume class and Cheeger isoperimetric constant]{Bounded volume class and Cheeger isoperimetric constant for negatively curved manifolds}
\author{Ervin Had\v{z}iosmanovi\'c}
\address{Scuola Normale Superiore, Piazza dei Cavalieri 7, 56126 Pisa, Italy}
\email{ervin.hadziosmanovic@sns.it}

\begin{document}

\maketitle

\begin{abstract}
We prove that for manifolds with negative curvature bounded away from $0$ of infinite volume and bounded geometry, the bounded fundamental class, defined via integration of the volume form over straight top-dimensional simplices, vanishes if and only if the Cheeger isoperimetric constant is positive. This gives a partial affirmative answer to a conjecture of Kim and Kim. Furthermore, we show that for all manifolds with negative curvature bounded away from $0$ of infinite volume, the positivity of the Cheeger constant implies the vanishing of the bounded volume class, solving one direction of the conjecture in full generality.
\end{abstract}

\section{Introduction}
In this paper, all manifolds will be assumed to be Riemannian, orientable, connected, complete and without boundary if not otherwise stated.
For an $n$-manifold $M$, being non-compact is equivalent to the vanishing of its top degree cohomology, that is $H^n(M,\R)=0$. This condition is equivalent to the vanishing of the de Rham class of the volume form $\omega$ induced by the metric. If the manifold has negative curvature bounded away from $0$, there is a well defined \textit{bounded volume class} $[\hat{\omega}]\in H_b^n(M,\R)$ in the $n$-dimensional bounded cohomology of $M$, which is defined as the class of the cocycle which corresponds to integration of the volume form on straight simplices (see \cref{bounded_volume_class} for a precise definition). One can ask whether it is possible to obtain a characterisation of its vanishing in terms of geometric properties. This problem was first tackled by Soma \cite{soma1997bounded}, in his study of bounded cohomology of Kleinian groups in degree $3$. In the context of infinite volume hyperbolic $3$-manifolds, he provides a complete answer in terms of \textit{geometric finiteness}:

\begin{teo}[{\cite[Theorem 1]{soma1997bounded}}]\label{soma}
	Let $M$ be an infinite volume hyperbolic $3$-manifold with finitely generated fundamental group. Then the bounded fundamental class $[\hat{\omega}]\in H^{3}_b(M,\R)$ vanishes if and only if $M$ is geometrically finite.
\end{teo}

Soma actually proved this result for \textit{topologically tame} hyperbolic $3$-manifolds, that is manifolds which are homeomorphic to the interior of a compact $3$-manifold. By the solution of the tameness conjecture by Agol \cite{agol2004tameness} and Calegari-Gabai \cite{calegari2006shrinkwrapping}, all hyperbolic $3$-manifolds with finitely generated fundamental group are topologically tame.
Geometric finiteness is a property that can be defined for any pinched negatively curved manifold, that is with sectional curvatures which lie between two negative constants (see for example \cite{bowditch1995geometrical}).

From a result of Hamenst\"adt \cite{hamenstadt2004small} it follows that for a pinched negatively curved manifold $M$, being geometrically finite implies that its Cheeger isoperimetric constant $h(M)$ does not vanish, which roughly means that there is a linear isoperimetric inequality for domains in $M$ (see \cref{cheegeriso}).
By a result of Bonahon \cite{bonahon1986bouts}, geometrically infinite hyperbolic 3-manifolds have a simply degenerate end, which implies that $h(M)=0$. Thus, for hyperbolic $3$-manifolds of infinite volume, being geometrically finite is equivalent to having positive Cheeger isoperimetric constant. Therefore, Soma's result can be formulated as follows:

\begin{prop}\label{soma2}
	Let $M$ be a hyperbolic $3$-manifold of infinite volume with finitely generated fundamental group. Then the bounded fundamental class $[\hat{\omega}]\in H^{3}_b(M,\R)$ vanishes if and only if $h(M)>0$.
\end{prop}

Kim and Kim \cite{kim2015bounded} generalise this result in higher dimensions to locally symmetric spaces of rank one (which include hyperbolic manifolds):

\begin{teo}[{\cite[Theorem 1.2]{kim2015bounded}}]\label{kimkim}
	Let $M$ be an $\R$-rank one locally symmetric space of infinite volume of dimension at least $3$. Then the bounded fundamental class of $M$ vanishes if and only if $h(M)>0$.
\end{teo}
Moreover, by extending Soma's results, they prove the same for more general negatively curved $3$-manifolds, under the hypotheses of bounded geometry and pinched negative curvature (see \cref{pnc} and \cref{bounded_geo}):
\begin{teo}[{\cite[Theorem 1.4]{kim2015bounded}}]\label{kimkim2}
	Let $M$ be an infinite volume, pinched negatively curved $3$-manifold with bounded geometry. Then the bounded fundamental class $[\hat{\omega}]\in H^{3}_b(M,\R)$ vanishes if and only if $h(M)>0$.
\end{teo}

Having these results in mind, Kim and Kim formulate the following conjecture:
\begin{conj}[\cite{kim2015bounded}]\label{conjecture}
	Let $M$ be an infinite volume, pinched negatively curved manifold of dimension $n$. Then the bounded fundamental class $[\hat{\omega}]\in H^{n}_b(M,\R)$ vanishes if and only if $h(M)>0$.
\end{conj}

As said before, geometrically finite, pinched negatively curved manifolds of infinite volume have positive Cheeger isoperimetric constant, and thus they provide a class of manifolds on which to test one of the directions of the conjecture. In this spirit, Kim and Kim obtain the following:
\begin{teo}\label{geom_fin}
	Let $M$ be an infinite volume, pinched negatively curved manifold of dimension $n$ at least $3$. If $M$ is geometrically finite, then its bounded fundamental class $[\hat{\omega}]\in H^{n}_b(M,\R)$ vanishes.
\end{teo}

\subsection*{Statement of the results}
The first main result of this paper generalises \cref{kimkim2} to all dimensions greater than $3$ and to manifolds with negative curvature bounded away from $0$, not necessarily pinched. It thus provides a positive answer to the conjecture for this class of manifolds:

\begin{mainteo}\label{teo1}
	Let $M$ be an infinite volume Riemannian manifold of dimension at least $3$ with negative curvature bounded away from $0$ and with bounded geometry. Then the bounded fundamental class of $M$ vanishes if and only if $h(M)>0$.
\end{mainteo}
Moreover, we show that one of the two implications of the conjecture holds even when the assumption of bounded geometry is dropped, and thus is true for all manifolds with negative curvature bounded away from $0$, generalising \cref{geom_fin}.

\begin{mainteo}\label{teo2}
	Let $M$ be an infinite volume Riemannian manifold of dimension at least 3 with negative curvature bounded away from $0$. If $h(M)>0$ then the bounded fundamental class of $M$ vanishes.
\end{mainteo}

In order to prove these results, we use techniques which are quite different from those of Soma and Kim and Kim. Indeed, their results apply to either symmetric spaces of any dimension or to $3$-manifolds with bounded geometry. For the first case, Kim and Kim use tools only available for locally symmetric spaces, while for the second case, the precise behaviour of $3$-manifolds at infinity is used (which relies on some highly non-trivial theorems). Instead, we use tools available in every dimension which do not need the knowledge of the geometry at infinity and also do not require that the manifold is a locally symmetric space.
For the first result, we exploit the construction of a bounded primitive of the volume form given by Sikorav \cite[Theorem 1.1]{sikorav2001growth}. For the second result, using the theory of currents on manifolds, we prove that a linear isoperimetric inequality for relatively compact domains of $M$ (which is equivalent to $h(M)>0$) implies a similar inequality for masses of straight singular chains which in turn implies that the bounded fundamental class vanishes.

\subsection*{Bounded primitives of the volume form}
The positivity of Cheeger's isoperimetric constant is linked to another property of the manifold, namely the existence of a bounded primitive of the volume form. A differential form is said to be bounded if it attains uniformly bounded values on orthonormal frames (see \cref{formsnorm}). If the volume form of a Riemannian manifold admits a bounded primitive, then it can be easily seen that the Cheeger isoperimetric constant must be positive (see \cref{primitive-cheeger}). Sullivan \cite{sullivan1976cycles} asked whether the converse is also true and Gromov \cite{gromov1981hyperbolic} asserted that the answer to Sullivan's question should be positive, without giving a detailed proof. As far as we know, the question is still open in full generality, and has been solved in two cases, namely when $M$ is a locally symmetric space \cite[Theorem 1.1]{kim2015bounded} or when $M$ has bounded geometry \cite[Theorem 1.1]{sikorav2001growth}, \cite[Remark in Section 3]{block1992aperiodic}. If the volume form $\omega$ admits a primitive $\alpha$, then the corresponding (bounded) cocycle $\hat{\omega}$, defined by integration, has the cocycle $\hat{\alpha}$ (also defined by integration) as a primitive. If $\alpha$ is bounded then $\hat{\alpha}$ also is (as a cocycle) and thus $[\hat{\omega}]=0$ in $H_n^b(M,\R)$ (see also \cref{boundedprim}). Thus, if $M$ has bounded geometry or is a locally symmetric space, one direction of \cref{conjecture} is true, and indeed this is the approach that Kim and Kim \cite{kim2015bounded} use to prove this implication. On the other hand, if $M$ is a locally symmetric space of real rank $1$ (for example a hyperbolic manifold), \cref{kimkim} implies that the vanishing of the bounded fundamental class is enough to ensure that $\omega$ has a bounded primitive $\alpha$. In some sense, the bounded cocycle $\phi$ can be upgraded to an actual differential form. This is the strategy with which Kim and Kim prove the conjecture for this class of manifolds. The structure of our argument for the proof of \cref{teo1} is similar (though the technique, as said above, is different), showing that bounded cocycles can be upgraded to bounded differential forms when the manifold has bounded geometry. We do not know whether this happens without the bounded geometry assumption, but an affirmative answer would solve the missing direction of \cref{conjecture} in full generality.

\subsection*{Plan of the paper}
In \cref{preliminari} we introduce some preliminary concepts which will be needed in the paper. In \cref{bounded_geo_case} we prove \cref{teo1} solving the conjecture in the case with bounded geometry, and in \cref{general_case} we prove \cref{teo2} in the general case.

\subsection*{Acknowledgement}
I would like to thank my advisor Roberto Frigerio for suggesting this problem to me and for reading the preliminary version of this paper. I would also like to thank Luigi Ambrosio and Francesco Milizia for useful conversations and for pointing out some references. Furthermore, I would  like to thank the anonymous referee for the careful reading of the paper and the useful comments. Finally, this work has been written within the activities of the GNSAGA group of INdAM (National Institute of Higher Mathematics).

\section{Preliminaries}\label{preliminari}
In this section we recall some preliminary notions and set the notation.
\subsection{Bounded cohomology.}
Bounded cohomology is a functional-analytic variant of singular cohomology, which was first introduced by Johnson \cite{johnson1972cohomology} and Trauber in the context of Banach algebras. It was then defined for topological spaces by Gromov in his seminal article \cite{gromov1982volume}. Let $X$ be a topological space. We denote by $C_k(X,\R)$ the set of singular chains with real coefficients. There is a natural norm that can be defined on this space, which is the $\ell^1$-norm corresponding to the basis given by singular simplices.
\begin{defn}
	Let $c=\sum_{i=1}^{n}a_i\sigma_i$ be a reduced singular chain (i.e., without cancellations among the simplices). We define 
	\[\norm{c}[1]=\norm{\sum_{i=1}^{n}a_i\sigma_i}[1]\coloneqq \sum_{i=1}^{n}|a_i|.\]
\end{defn}
This norm turns $C_k(X,\R)$ into a normed vector space. When defining singular cohomology, one takes the algebraic dual of this vector space. To define bounded cohomology, one takes instead the topological dual of $(C_k(X,\R),\norm{\cdot}[1])$, i.e., the space of \textit{bounded} cochains $C_b^k(X,\R)$. More explicitly, we have a norm on $C^k(X,\R)$ given by
\[\norm{\phi}[\infty]\coloneqq\sup_{\norm{c}[1]\le 1}|\phi(c)|\in [0,+\infty],\]
and $C_b^k(X,\R)$ is the subspace of $C^k(X,\R)$ consisting of cochains of finite norm.
The usual differential of cochains restricts to bounded ones and thus one gets a cochain complex $C_b^*(X,\R)$.
\begin{defn}
	The bounded cohomology of $X$, denoted by $H_b^*(X,\R)$, is the cohomology of the cochain complex $C_b^*(X,\R)$.
\end{defn}
Some properties of classical cohomology are still valid in this context (such as invariance up to homotopy), but others are not. The most striking one is the absence of excision, which makes bounded cohomology behave much differently (see \cite{frigerio2017bounded} for an introduction on the subject). Trying to compute bounded cohomology groups is not a trivial task. Even to show whether they vanish or not is a difficult problem in general, and thus it is interesting to try to build non-trivial classes.

\subsection{Bounded volume class}\label{bounded_volume_class}
\begin{defn}\label{pnc}
	Let $M$ be a complete Riemannian manifold. We say that $M$ has \textit{negative curvature bounded away from $0$} if there exists $\epsilon>0$ such that $k\le-\epsilon$ for all sectional curvatures $k$. If in addition there exists $c>0$ such that $-c<k<-\epsilon$, then we say that $M$ is \emph{pinched negatively curved} (PNC).
\end{defn}
Let now $M$ be an $n$-dimensional negatively curved manifold with $n\ge 2$. The universal cover $\tilde{M}$ is diffeomorphic to $\R^n$ and is \textit{uniquely continuously geodesic} (i.e., any pair of points is joined by a unique geodesic, and geodesics depend continuously on their endpoints), which implies that there is a way to ``straighten'' singular simplices. We recall briefly this construction, see \cite[Section 8.4]{frigerio2017bounded} for more details. First, one defines straight simplices on the universal cover $\tilde{M}$ by induction on the dimension: a $0$-dimensional straight simplex is just a map from a point to a given point $x_0\in \tilde{M}$; given $k+1$ points $x_0,\ldots,x_k\in \tilde{M}$, one defines the straight simplex $[x_0,\ldots,x_k]$ by coning $[x_0,\ldots,x_{k-1}]$ over $x_k$, i.e., by taking constant speed geodesics from $x_k$ to the support of the $(k-1)$-dimensional straight simplex $[x_0,\ldots,x_{k-1}]$. Then, given a simplex $\sigma:\Delta^k\to M$, one chooses a lift $\tilde{\sigma}:\Delta^k\to \tilde{M}$, considers the straight simplex with the same vertices as $\tilde{\sigma}$, and projects it to $M$. By extending this procedure to chains, we obtain a well defined chain map 
\[\operatorname{str}_*:C_*(M,\R)\to C_*(M,\R)\]
whose image consists of straight chains and which is chain homotopic to the identity via a bounded homotopy. Therefore, the induced map in bounded cohomology
\begin{align*}
	H_b^*(M,\R)&\longrightarrow H_b^*(M,\R)\\
	[\phi]&\longmapsto [\phi\circ\operatorname{str}]
\end{align*}
is the identity. This means that to get a well defined class in bounded cohomology, it is sufficient to define a bounded cochain on straight singular simplices.
Let now $\omega$ be the Riemannian volume form of $M$. We can define the associated integration cochain by setting \[\hat{\omega}(c)=\int_{\operatorname{str}(c)}\omega\]
for any $c\in C_n(M,\R)$. This defines a bounded singular cocycle. Indeed, by Stokes' theorem and since $\omega$ is closed, we have
\[\delta\hat{\omega}(c)=\hat{\omega}(\partial c)=\int_{\operatorname{str}(\partial c)}\omega=\int_{\partial\operatorname{str}(c)}\omega=\int_{\operatorname{str}(c)}d\omega=0\]
for any chain $c$. Moreover, it is bounded since for any straight simplex $\sigma$
\[|\hat{\omega}(\sigma)|=\vol(\sigma),\]
and straight $k$-simplices have uniformly bounded volume for $k\ge2$ \cite{inoue1982gromov} when the curvature is negative bounded away from $0$.

\subsection{Cheeger isoperimetric constant and bounded primitives}
Let $M$ be a complete Riemannian manifold of infinite volume.
\begin{defn}\label{cheegeriso}
	The \textit{Cheeger isoperimetric constant} of $M$ is defined as
	\begin{align*}
		h(M)\coloneqq\inf\limits_{U\subseteq M} \frac{\mathrm{vol}(\partial U)}{\mathrm{vol}(U)}
	\end{align*}
	where $U$ ranges over open submanifolds of $M$ with $\overline{U}$ compact and $\partial U$ smooth, and $\vol(\partial U)$ denotes the $(n-1)$-dimensional volume of $\partial U$.
\end{defn}
If $h(M)>0$, it follows directly from the definition that for any relatively compact open submanifold $U$ we have a linear isoperimetric equality:
\[\vol(U)\le\frac{1}{h(M)}\vol(\partial U).\]

The Euclidean space $\R^n$ has $h(\R^n)=0$ since $\vol(B_r)\sim r^n$ while $\vol(\partial B_r)\sim r^{n-1}$. Instead, hyperbolic space $\HH^n$ has $h(\HH^n)>0$ and more generally simply connected manifolds with negative curvature bounded away from $0$ have $h(M)>0$ \cite{brooks1983some}. For a non-simply connected example, one can take $M$ to be a hyperbolic $3$-manifold fibering over the circle with fiber a hyperbolic surface $S$ and take the infinite cyclic cover $S\times \R$ associated to $\pi_1(S)<\pi_1(M)$ with the induced metric. It can be seen that the volume of $S\times [0,n]$ grows linearly in $n$ while the area of $S\times\{n\}$ remains bounded, implying that $h(M)=0$.

Using the Riemannian structure on $M$ it is possible to define a norm on differential forms.
\begin{defn}\label{formsnorm}
	Let $\alpha\in \Omega^k(M)$ be a $k$-differential form on $M$. Its $L^{\infty}$-\textit{norm} is defined as
	\begin{align*}
		\|\alpha\|_{\infty}\coloneqq\sup\limits_{x\in M} \sup\limits_{(v_1\dots v_k)\in (T_x M)^k} \frac{\alpha(v_1,\dots,v_k)}{\|v_1\|_x\dots\|v_k\|_x}.
	\end{align*}
	A differential form $\alpha$ is said to be \textit{bounded} if $\norm{\alpha}[\infty]<\infty$.
\end{defn}
By definition, the volume form $\omega$ satisfies $\norm{\omega}[\infty]=1$.
If $M$ is of infinite volume, then it is open and $H_{\textrm{dR}}^n(M)=0$. Thus the volume form $\omega$ must have a primitive $\alpha\in \Omega^{n-1}(M)$. 

\begin{prop}\label{primitive-cheeger}
	Let $M$ be any Riemannian manifold. If the volume form $\omega$ admits a bounded primitive $\alpha$, then $h(M)>0$.
\end{prop}
\begin{proof}
	Let $U\subseteq M$ be a relatively compact open submanifold with smooth boundary. By Stokes' theorem we get
	\begin{align*}
		\vol(U)=\int_U\omega=\int_Ud\alpha=\int_{\partial U}\alpha\le \|\alpha\|_{\infty}\vol(\partial U),
	\end{align*}
	which implies that $h(M)\ge \|\alpha\|_{\infty}^{-1}$.
\end{proof}
\begin{oss}\label{inverse_bounded}
	As said in the introduction, whether the converse of the above proposition holds is still open in full generality, and is known when $M$ is a locally symmetric space \cite[Theorem 1.1]{kim2015bounded} or when $M$ has bounded geometry (\cite[Theorem 1.1]{sikorav2001growth} or \cite[Remark in Section 3]{block1992aperiodic}).
\end{oss}

\begin{defn}\label{bounded_geo}
	A complete Riemannian manifold $M$ has \emph{bounded geometry} if its sectional curvatures are bounded in absolute value and its injectivity radius is positive. 
\end{defn}
Examples of manifolds with bounded geometry are PNC manifolds with positive injectivity radius and covers of compact Riemannian manifolds.

It can be easily shown that if the volume form has a bounded primitive then the bounded fundamental class vanishes:

\begin{prop}\label{boundedprim}
	Let $M$ be an infinite volume, Riemannian manifold of dimension at least $3$ with negative curvature bounded away from $0$. If its volume form $\omega$ admits a bounded primitive, then $[\hat{\omega}]=0$ in $H_b^n(M,\R)$.
\end{prop}

\begin{proof}
	Let $\omega=d\alpha$ with $\alpha$ bounded. One can define the singular $(n-1)$-cochain $\hat{\alpha}$ by setting 
	\[\hat{\alpha}(c)=\int_{\operatorname{str}(c)}\alpha.\] 
	By hypothesis, $n-1\ge 2$ and thus volumes of straight $(n-1)$-simplices are uniformly bounded. Since $\alpha$ is bounded, it follows that $\hat{\alpha}\in C_b^{n-1}(M,\R)$. Now, by Stokes' Theorem, we have $\hat{\omega}=\delta \hat{\alpha}$ as cocycles and thus $[\hat{\omega}]=0$.
\end{proof}

\cref{inverse_bounded}, and \cref{boundedprim} provide a partial answer for one direction of \cref{conjecture}.

\begin{prop}\label{onedirection}
	Let $M$ be an infinite volume Riemannian manifold with negative curvature bounded away from $0$, which is either with bounded geometry or an $\R$-rank one locally symmetric space of dimension at least $3$. If $h(M)>0$, then the bounded fundamental class of $M$ vanishes.
\end{prop}

\section{The bounded geometry case}\label{bounded_geo_case}
In this section, we prove \cref{conjecture} in the case of negatively curved manifolds with bounded geometry. This generalises the statement of \cref{kimkim2} above, which holds for PNC $3$-manifolds with bounded geometry.
From \cref{onedirection}, we already know one implication. It remains to show that the vanishing of the bounded fundamental class implies that $h(M)>0$. To do so, we use the construction of a bounded primitive given by Sikorav \cite{sikorav2001growth} under the hypothesis of $h(M)>0$ and bounded geometry. We recall the result and give a sketch of its proof, since it will be useful for our purposes.

\begin{teo}[{\cite[Theorem 1.1]{sikorav2001growth}}]\label{boundedprim2}
	Let $M$ be an infinite volume Riemannian manifold of bounded geometry. If $h(M)>0$ then the volume form $\omega$ admits a bounded primitive.
\end{teo}
\begin{proof}[Sketch of proof]
	From the fact that $M$ has bounded geometry, one can find a \textit{triangulation with bounded geometry}, which means that the link of each vertex contains a uniformly bounded amount of simplices, each $k$-simplex $\sigma$ is the image of a diffeomorphism $\sigma:\Delta^k\to M$ such that its Jacobian has bounded norm (such that the constant does not depend on the simplex), and this map can be extended to a neighbourhood of $\Delta^k$ with the same property (see \cite{sikorav2001growth} and \cite{attie1994quasi}).
	Let $K$ be such a triangulation on $M$. From $h(M)>0$ (which implies a linear isoperimetric inequality for domains), it can be shown that for any simplicial chain $T\in C_n(K,\R)$ one has $|\hat{\omega}(T)|\le C \|\partial T\|_1$ for some $C>0$ (hence, in particular, $\hat{\omega}$ is bounded on simplicial chains). From this inequality, via a Hahn-Banach argument, it is possible to define a bounded simplicial cocycle $\phi\in C_b^{n-1}(K,\R)$ such that $\hat{\omega}=\delta\phi$ on simplicial chains.
	Let $I^*:\Omega^*(M)\to C^*(K,\R)$ be the chain map defined by integration of differential forms over simplices of the triangulation. It is possible to define a right inverse of $I^*$ as follows. For each vertex $v\in K^{(0)}$ of the triangulation, consider the open set $\operatorname{st}(v)$ given by the interior of the union of all $n$-simplices containing $v$. Let $\{g_v:v\in K^{(0)}\}$ be a partition of unity subordinated to the open cover $\{\operatorname{st}(v):v\in K^{(0)}\}$. Let $1\le k \le n$ and for each $k$-simplex $\sigma$, denote by $\sigma^*\in C^k(M,\R)$ the dual element, that is, the cochain which takes  value $1$ on $\sigma$ and 0 otherwise. If $\sigma$ is given by its  vertices $v_0,\dots,v_k$ one defines
	\[P^k(\sigma^*)=k!\sum_{i=0}^{k}g_{v_i}dg_{v_0}\wedge\cdots\wedge\widehat{dg_{v_i}}\wedge\cdots\wedge dg_{v_k},\]
	and for any cochain $T\in C^k(K,\R)$
	\[P^k(T)=\sum_{\sigma \in K^{(k)}}T(\sigma)P^k(\sigma^*).\]
	One can check that this defines a right inverse of $I^*$ \cite[Section 6.2, Lemma 1]{singer2015lecture} and, since $K$ is of bounded geometry, it can be seen that the quantities $\norm{dg_v}[\infty]$ are uniformly bounded and thus $P^*$ sends bounded cochains to bounded differential forms.
	
	By the definition of $\hat{\omega}$, we have that $I^n(\omega)=\hat{\omega}$ on simplicial chains, hence
	\[P^nI^n(\omega)=P^n\delta\phi=d P^{n-1}(\phi),\]
	where $P^{n-1}(\phi)$ is a bounded differential form. The integral of the $n$-form $\omega-P^nI^n(\omega)$ vanishes on every $n$-simplex of $K$ and thus, following through the constructive proof of de Rham's Theorem given in \cite{singer2015lecture} while keeping track of the norms (using again that $K$ has bounded geometry), one can get a bounded primitive $\eta$. Thus we have \[\omega=P^nI^n(\omega)+d\eta=dP^{n-1}(\phi)+d\eta=d(P^{n-1}(\phi)+\eta),\]
	and $P^{n-1}(\phi)+\eta$ is a bounded primitive of $\omega$.
\end{proof}

We can now prove the following theorem, which implies \cref{teo1}. The structure of this result is the same as \cite[Theorem 1.2]{kim2015bounded} but with more general hypotheses. 
\begin{teo}
	Let $M$ be an infinite volume Riemannian manifold of dimension at least $3$ with negative curvature bounded away from $0$ and bounded geometry. Then the following are equivalent:
	\begin{enumerate}
		\item $h(M)>0$,
		\item the volume form $\omega$ has a bounded primitive,
		\item the bounded fundamental class $[\hat{\omega}]$ vanishes.
		
	\end{enumerate}
\end{teo}

\begin{proof}
	The equivalence of $1)$ and $2)$ is given by \cref{primitive-cheeger} and \cref{boundedprim2}, while the implication $2)\implies 3)$ is \cref{boundedprim}.
	Thus, it remains to prove that $3)\implies 2)$. This follows easily from  Sikorav's construction of a bounded primitive sketched in the proof of \cref{boundedprim2}. In fact, by hypothesis we have that $[\hat{\omega}]=0$ in $H_b^n(M,\R)$ which means that there exists a bounded cochain $\phi\in C_b^{n-1}(M,\R)$ such that $\hat{\omega}=\delta\phi$. Now, we fix a triangulation with bounded geometry $K$ and after restricting the cochains $\hat{\omega}$ and $\phi$ on the simplices of this triangulation, we get bounded simplicial cochains which fit exactly in the situation of the proof of \cref{boundedprim2}. Following this proof, we find the desired bounded primitive of $\omega$ as above.
\end{proof}
\section{The general case}\label{general_case}

In this section, we remove the assumption that $M$ has bounded geometry and prove:
\begin{teo}\label{Teo}
	Let $M$ be a Riemannian manifold of dimension $n$ at least $3$, with infinite volume and with negative curvature bounded away from $0$. If $h(M)>0$, then $[\hat{\omega}]=0$ in $H_b^n(M,\R)$.
\end{teo}

To prove this result, we will need to introduce the concept of currents and functions of bounded variation on a manifold.
\subsection{Currents and functions of bounded variation}\label{currents}
Let $M$ be an $n$-dimensional Riemannian manifold. A $k$-\textit{current} on $M$ is a bounded linear functional on the space $\Omega_c^k(M)$ of compactly supported smooth $k$-differential forms. Given a $k$-current $T$ one can define its \textit{mass} $\operatorname{M}(T)$ by setting 
\[\mathrm{M}(T)=\sup_{\substack{\phi\in \Omega_c^k(M) \\\|\phi\|_\infty\le 1}} |T(\phi)|.\]
Moreover, the \textit{boundary} of a $k$-current $T$ is a $(k-1)$-current $\partial T$ defined by $\partial T(\phi)=T(d\phi)$ for any $\phi\in \Omega_c^{k-1}(M)$. A current $T$ is of \textit{finite mass} if $\operatorname{M}(T)<\infty$ and is \textit{normal} if both $\operatorname{M}(T)$ and $\operatorname{M}(\partial T)$ are finite.

Let $f\in L^1(M)$. Its \emph{variation}, denoted by $|Df|(M)$, is defined as 
\[|Df|(M)\coloneqq \sup_{\substack{X\in \Gamma_c(TM) \\\|X\|\le 1}} \left\lvert\int_{M}f\operatorname{div}X\omega\right\lvert,\]
where $\Gamma_c(TM)$ denotes the space of compactly supported smooth vector fields on $M$.
If $|Df|(M)$ is finite, $f$ is said to be of \emph{bounded variation} and the vector space of all such functions is denoted by $\operatorname{BV}(M)$ (see \cite[Section 1]{miranda2007heat} for all relevant definitions and basic properties of such functions on a Riemannian manifold, and also \cite[Section 3]{ambrosio2000functions} for the classical Euclidean case).

\begin{oss}\label{variation}
The definition we give of variation is apparently slightly different from the one given in \cite[Section 1.2]{miranda2007heat}. In that paper, the definition is given by using compactly supported smooth differential $1$-forms in the integral, but the definition is the same since $\Omega_c^1(M)$ is naturally isometrically isomorphic to $\Gamma_c(TM)$ via the non-degenerate pairing induced by the metric of $M$. Moreover, one can also compute $|Df|(M)$ using compactly supported differential $(n-1)$-forms. Indeed, given any smooth vector field $X\in \Gamma(TM)$, one can consider the contraction of the volume form $\omega$ with $X$, denoted by $\iota_X\omega$.
This defines a smooth $(n-1)$-differential form and $\iota_\bullet\omega$ is an isometric isomorphism between vector fields and $(n-1)$-forms, which also preserves the property of having compact support. By definition of divergence, one has $(\operatorname{div} X)\omega=d(\iota_X\omega)$. Thus, we have the following formula for $|Df|(M)$:

\[|Df|(M)=\sup_{\substack{\alpha\in \Omega_c^{n-1}(M) \\\|\alpha\|_\infty\le 1}}\left\lvert\int_{M}fd\alpha\right\lvert.\]
\end{oss}
A consequence of the Riesz-Markov-Kakutani representation theorem \cite[Theorem 2.14]{rudin1987real} is that any $k$-current of finite mass $T$ can be represented by a nonnegative measure $\mu$ on $M$ and an $L^1(\mu)$ $k$-vector field $S:M\to \Lambda^kTM$ such that $\norm{S_x}[]=1$ for any $x\in M$, and for any $\phi\in \Omega_c^k(M)$ it holds that
\[T(\phi)=\int_{M}\langle\phi(x),S(x)\rangle d\mu(x),\]
where $\langle\cdot,\cdot\rangle$ denotes the duality pairing between $\Lambda^kT^*M$ and $\Lambda^kTM$. Moreover, $\operatorname{M}(T)=|\mu|(M)$.
The case $k=n$ has an even simpler description. In fact, one can take $S$ to be the unitary $n$-vector field obtained by gluing local orthonormal frames on $M$ for which, by definition, $\langle\omega,S\rangle=1$ holds. Since every smooth $n$-form $\phi\in \Omega^n_c(M)$ can be written as $g_{\phi}\omega$ for some $g_{\phi}\in C_c^\infty(M)$, it holds that $\langle\phi,S\rangle=\langle g_{\phi}\omega,S\rangle=g_{\phi}$ and thus

\[T(\phi)=\int_{M}g_{\phi}(x)d\mu(x).\]

If in addition $T$ is a normal $n$-current, then the measure $\mu$ is absolutely continuous with respect to the volume measure given by $\omega$ and its density is a function of bounded variation $f\in \operatorname{BV}(M)$.
It follows that the image $T(\phi)$ of $\phi\in \Omega_c^n(M)$ under $T$ is given by
\[T(\phi)=\int_{M}fg_{\phi}\omega=\int_Mf\phi\]
and the image $\partial T(\alpha)$ of $\alpha\in \Omega^{n-1}_c(M)$ by 

\[\partial T(\alpha)=T(d\alpha)=\int_{M}fd\alpha.\]

From the equality $\operatorname{M}(T)=|\mu|(M)$, one gets that $\operatorname{M}(T)=\norm{f}[L^1]$.
Moreover, one also has that $\operatorname{M}(\partial T)=|Df|(M)$ by \cref{variation}. 
\subsection{Proof of \cref{Teo}}
We will now consider normal currents associated to straight chains. Let $M$ be a Riemannian manifold with negative curvature bounded away from $0$. 
Let $c=\sum_i a_i\sigma_i\in C_k(M,\R)$ be a straight chain and $\phi\in \Omega_c^k(M)$. We can define the integral of $\phi$ on $c$ by setting

\[\int_c\phi=\sum_ia_i\int_{\sigma_i}\phi.\]

With this definition, we get a current, which we will denote by $T_c$. By Stokes' Theorem we easily get that $\partial T_c=T_{\partial c}$. Indeed, for any $(k-1)$-form $\phi$,
\[\partial T_c(\phi)=T_c(d\phi)=\int_{c}d\phi=\int_{\partial c}\phi=T_{\partial c}(\phi).\]
For these currents, we will put $\operatorname{M}(c)\coloneqq \operatorname{M}(T_c)$ and we will refer to it as the mass of the straight chain.
By definition, if $\|\phi\|_\infty\le1$ we have
\[|T_c(\phi)|=\left|\int_c\phi\right|=\left|\sum_ia_i\int_{\sigma_i}\phi\right|\le \sum_i|a_i|\left|\int_{\sigma_i}\phi\right|\le\sum_i|a_i|\vol{\sigma_i}.\]
Thus, if $k\ge 2$, $T_c$ is of finite mass. Indeed, volumes of straight $k$-simplices are uniformly bounded by a constant $v_k>0$ and thus by taking the supremum over $k$-forms we get
\[\mathrm{M}(c)\le \|c\|_1v_k<\infty.\]
If in addition $k\ge 3$, the current $T_c$ is also normal (since the same bound as above also holds for the $(k-1)$-chain $\partial c$). 
Let now $c\in C_n(M,\R)$ be a top-dimensional straight chain. Since the support of $c$ is compact, we can find a compactly supported $n$-form $\tilde{\omega}$ such that $\omega=\tilde{\omega}$ on an open set containing it. We have
\[|\hat{\omega}(c)|=\left|\int_c\omega\right|=\left|\int_c\tilde{\omega}\right|\le \mathrm{M}(c).\]
We can now prove the following proposition, which is an application of the Hahn-Banach theorem.
\begin{prop}
	Let $M$ be a Riemannian manifold of dimension $n$ at least $3$, with infinite volume and negative curvature bounded away from $0$. Suppose that the following inequality holds for singular chains: there exists $C>0$ such that for any $n$-dimensional straight chain $c\in C_n(M,\R)$ we have
	
	\[\mathrm{M}(c)\le C \mathrm{M}(\partial c).\]
	Then, $[\hat{\omega}]=0$ in $H_b^n(M,\R)$.

\end{prop}
\begin{proof}
	We look for a bounded cocycle $\psi\in C_b^{n-1}(M,\R)$ such that $\delta\psi=\hat{\omega}$.
	First, we define such a cocycle on the space $B_{n-1}(M,\R)$ of $(n-1)$-boundaries. To get a primitive of $\hat{\omega}$, we are forced to set
	
	\[\psi(\partial c)=\hat{\omega}(c)\quad \forall c\in C_{n}(M,\R).\]
	This cochain is well defined. Indeed, since $M$ is non-compact, there exists an $(n-1)$-form $\alpha$ such that $\omega=d\alpha$ and
	if we take another chain $c'$ such that $\partial c=\partial c'$, we have that
	\[\int_{c-c'}\omega=\int_{c-c'}d\alpha=\int_{\partial c-\partial c'}\alpha=0.\]
	Now, we extend this functional in order to get a bounded one defined on the whole space of chains $C_{n-1}(M,\R)$. To do so, we use the Hahn-Banach Theorem, which requires $\psi$ to be bounded on $B_{n-1}(M,\R)$. From the hypothesis and the properties of the mass of a simplex stated above, for any boundary $\partial c$ we have
	
	\[|\psi(\partial c)|=|\hat{\omega}(c)|\le\mathrm{M}(c)\le C\mathrm{M}(\partial c)\le Cv_{n-1}\|\partial c\|_1,\]
	which means exactly that $\psi$ is bounded on boundaries. Thus, we can extend it and get an element $\psi\in C_b^{n-1}(M,\R)$. By definition we have
	\[\delta\psi(c)=\psi(\partial c)=\hat{\omega}(c),\]
	showing that $[\hat{\omega}]=0$ in $H_b^n(M,\R)$.
\end{proof}

By this proposition, in order to conclude the proof of \cref{Teo} it is sufficient to prove that an isoperimetric inequality on domains of $M$ (which is equivalent to the positivity of $h(M)>0$) implies an analogous inequality for masses of straight singular chains.

\begin{prop}
	Let $M$ be a Riemannian manifold with $h(M)>0$. Then, there exists $C>0$ such that for any singular straight chain $c\in C_n(M,\R)$ we have
	\[\mathrm{M}(c)\le C \mathrm{M}(\partial c).\]
\end{prop}
\begin{proof}
	
	By the coarea formula \cite[par. VIII.3]{chavel1995riemannian}, the condition $h(M)>0$ implies that there exists $C>0$ such that for any compactly supported smooth function $h$ one has the following Poincar\'{e} inequality:
	
	\[\norm{h}[L^1]\le C\norm{\nabla h}[L^1].\]
	
	Let $T_c$ be the normal current associated to the singular chain $c$. As seen in \cref{currents}, it can be represented by a function $f\in \operatorname{BV}(M)$ such that $\norm{f}[L^1]=\operatorname{M}(T_c)=\operatorname{M}(c)$. 
	
	By \cite[Prop. 1.4]{miranda2007heat}, for any $u\in \operatorname{BV}(M)$ there exists a sequence $\{u_n\}\subseteq C^\infty_c(M)$ such that $u_n\to u$ in $L^1(M)$ and such that
	\[\lim_{n\to \infty} \norm{\nabla u_n}[L^1]=|Du|(M),\]
	which implies that we have
	\[\norm{u}[L^1]\le C |Du|(M).\]
	Applying this inequality to the function $f$ representing the chain $c$, we get 
	\[\operatorname{M}(c)=\norm{f}[L^1]\le C|Df|(M)=C\operatorname{M}(\partial c),\]
	which concludes the proof.
\end{proof}
\bibliographystyle{plainurl}
\bibliography{bibliografia} 
\nocite{*}

\end{document}